\documentclass[a4paper,10pt]{article}

\def\longtitle{A Safe Computational Framework for Integer Programming applied to Chv\'atal's Conjecture}

\def\shortfunding{The work for this article has been conducted within the
  Research Campus MODAL funded by the German Federal Ministry of Education and
  Research (BMBF grant numbers 05M14ZAM, 05M20ZBM).}

\usepackage[english]{babel}
\usepackage[utf8]{inputenc}

\usepackage{geometry}

\usepackage{mathtools}
\usepackage{amsmath}
\usepackage{amssymb}
\usepackage{amsthm}
\usepackage{csvsimple}
\usepackage{graphicx}
\usepackage{booktabs}
\usepackage{url}
\usepackage[colorinlistoftodos]{todonotes}

\reversemarginpar
\newtheorem{theorem}{Theorem}
\newtheorem{corollary}{Corollary}
\newtheorem{proposition}{Proposition}
\newtheorem{conjecture}{Conjecture}
\newtheorem{observation}{Observation}

\usepackage{authblk}

\usepackage[]{titlesec}
\usepackage{etoolbox}
\makeatletter
\patchcmd{\ttlh@hang}{\parindent\z@}{\parindent\z@\leavevmode}{}{}
\patchcmd{\ttlh@hang}{\noindent}{}{}{}
\makeatother
\titleformat{\section}
{\normalfont\large\bfseries}{\thesection}{0.9ex}{}
\titleformat{\subsection}
{\normalfont\normalsize\bfseries}{\thesubsection}{0.9ex}{}
\titleformat{\subsubsection}
{\normalfont\normalsize\upshape}{\thesubsubsection}{0.9ex}{}
\titleformat{\paragraph}[runin]
{\normalfont\normalsize\itshape}{\theparagraph}{1em}{}
\titleformat{\subparagraph}[runin]
{\normalfont\normalsize\itshape}{\theparagraph}{1em}{}
\titlespacing*{\section}     {0pt}{21dd plus 8pt minus 4pt}{10.5dd}
\titlespacing*{\subsection}   {0pt}{21dd plus 8pt minus 4pt}{10.5dd}
\titlespacing*{\subsubsection}{0pt}{19dd plus 8pt minus 4pt}{10.5dd}
\titlespacing*{\paragraph}   {0pt}{13pt plus 8pt minus 4pt}{1em}
\titlespacing*{\subparagraph}   {0pt}{13pt plus 8pt minus 4pt}{1em}

\clearpage{}

\usepackage{amsmath}
\usepackage{amssymb}
\usepackage{bbm}
\usepackage{calc}
\usepackage{xspace}
\usepackage{xcolor}
\usepackage{collcell}

\newcommand{\ie}{i.e.,\xspace}

\newcommand{\wlogU}{W.l.o.g.\ }
\newcommand{\wlogL}{w.l.o.g.\ }

\newcommand{\bandb}{branch-and-bound\xspace}

\newcommand{\Abs}[1]{\vert #1\vert}

\newcommand{\incset}{\mathcal{S}}

\newcommand{\floor}[1]{\lfloor #1\rfloor}

\newcolumntype{R}{>{\collectcell\ApplyColor}{r}<{\endcollectcell}}

\newcommand{\N}{\mathbb{N}\xspace}

\newcommand{\reduction}[2]{\pgfmathparse{(#1)/#2 }{\pgfmathprintnumber[textnumber]{\pgfmathresult}}}

\newcommand{\colorquot}[2]{
	\pgfmathparse{(#1<0.9*#2)?1:0}\ifdim\pgfmathresult pt>0pt \textcolor{blue}{$\mathbf{\reduction{#1}{#2}}$}\else
	\pgfmathparse{(#1>1.1*#2)?1:0}\ifdim\pgfmathresult pt>0pt \textcolor{red}{$\mathbf{\reduction{#1}{#2}}$} \else
	$\reduction{#1}{#2}$\fi \fi
}

\usepackage{algorithm}
\usepackage{algpseudocode}

\newcommand{\optprob}[1]{$P_{\text{\itshape opt}}(#1)$}
\newcommand{\infprob}[1]{$P_{\text{\itshape inf}\,}(#1)$}
\newcommand{\redprob}[1]{P_{\text{\itshape red}}(#1)}

\usepackage{enumitem}

\usepackage{array}
\newcolumntype{$}{>{\global\let\currentrowstyle\relax}}
\newcolumntype{^}{>{\currentrowstyle}}

\usepackage{todonotes}
\begin{document}


\title{\Large\longtitle\footnote{\shortfunding}\bigskip}
\author[1]{Leon Eifler }
\author[1]{Ambros Gleixner}
\author[1,2]{Jonad Pulaj}
\affil[1]{Department of Mathematical Optimization, Zuse Institute Berlin (ZIB),
Berlin, Germany, {\tt \{eifler,gleixner,pulaj\}@zib.de}}
\affil[2]{Department of Mathematics and Computer Science, Davidson College, Davidson, NC \tt jonad.pulaj@cavehill.uwi.edu}
\setcounter{Maxaffil}{0}
\renewcommand\Affilfont{\itshape\small}
\maketitle

 \begin{abstract}
   We describe a general and safe computational framework that provides
   integer programming results with the degree of certainty that is required for
   machine-assisted proofs of mathematical theorems.
   At its core, the framework relies on a rational branch-and-bound certificate
   produced by an exact integer programming solver, SCIP, in order to circumvent
   floating-point roundoff errors present in most state-of-the-art solvers for
   mixed-integer programs.
   The resulting certificates are self-contained and checker software exists that
   can verify their correctness independently of the integer programming solver
   used to produce the certificate.
   This acts as a safeguard against programming errors that may be present in
   complex solver software.
   The viability of this approach is tested by applying it to finite cases of
   Chvátal's conjecture, a long-standing open question in extremal combinatorics.
   We take particular care to verify also the correctness of the input for this
   specific problem, using the Coq formal proof assistant.
   As a result we are able to provide a first machine-assisted proof that
   Chvátal's conjecture holds for all downsets whose union of sets contains seven
   elements or less.
 \end{abstract}



\section{Introduction}\label{intro}

The work on algorithms and software for mathematical optimization is often
motivated by the solution of real-world applications.
This sometimes overshadows the value that these methods can have for answering
questions in mathematics itself.
Some of them can quite naturally be cast in the form of optimization problems.
Prominent examples are the extensive use of linear programming to settle
Kepler's conjecture~\cite{hales2017} or the use of semidefinite programming in
discrete geometry~\cite{BachocVallentin2006}.
In addition, first attempts at using integer programming have been made to
settle open questions in extremal combinatorics~\cite{PulajThesis}, graph theory \cite{LanciaEtAl2020},
and graph pebbling \cite{KenterEtAl2018}.

The use of integer programming (IP) for constructing rigorous mathematical
proofs, however, is faced with two main computational difficulties.
First, virtually all state-of-the-art IP solvers rely on fast floating-point
arithmetic, hence their results are compromised by roundoff errors.
Second, most solvers do not provide sufficient output that would allow to check
and verify the correctness of their result.
These limitations are unfortunate given that the presence of integer variables
allows for expressive models and that solvers for integer programming problems
have strongly increased in computational power over the last years~\cite{AchterbergWunderling2013}.

In marked contrast, satisfiability solving (SAT) has been employed with
considerable success to answer open questions in discrete mathematics.
A recent milestone is the solution of the boolean Pythagorean triples
problem~\cite{heule2016solving}.
Satisfiability solving and integer programming share the theoretical difficulty
that compact certificates are in general not available since both SAT and IP are
not known to be in co-$\mathcal{NP}$.
However, over the last years the SAT community has established standards and
tools for proof logging and solver-independent verification of
results~\cite{WetzlerHeuleHunt2014}.

In comparison, exact IP software with verifiable results is still in its
infancy.
Besides domain-specific work such as for the traveling salesman solver
Concorde~\cite{applegate2006concorde} or partial functionality within software libraries targeted towards polyhedral analysis~\cite{Polymake2017,BagnaraHZ08SCP}, the only exact, general IP solver we are aware of
and whose results are not compromised by floating-point errors is an extension
of the solver SCIP~\cite{CookKochSteffyetal.2013}.
Exact SCIP has recently been further extended by the possibility to print
certificates that can be verified independently from the solution
process~\cite{VIPR}.
The goal of this paper is to demonstrate how these tools can be employed to
create a \emph{safe} computational framework for investigating a particular
mathematical application by integer programming.
To this end, we couple
\begin{enumerate}
\item an exact rational IP \emph{solver}, SCIP~\cite{CookKochSteffyetal.2013}, with
\item IP certificates for its branch-and-bound tree \emph{output} that can be
  verified independently from the solution process by the checker
  VIPR~\cite{VIPR},
\item and verification procedures for the correctness of the \emph{input}
  implemented in a formal proof assistant, Coq~\cite{coq2018}.
\end{enumerate}
Notably, this framework features the combined use of an exact IP certificate and
a formal proof assistant to ensure the correctness of the certificate's input
data. To the best of our knowledge, this has not been explored in the literature
before.

We apply this framework to Chvátal's conjecture, which is a well-known open
problem in extremal set theory dating back to 1974 and contained in Erd{\H{o}}s's
list of favorite combinatorial problems~\cite{erdHos1981combinatorial}.
Using our framework, we obtain machine-assisted proofs for low-dimensional cases
that were previously unknown.

The rest of this paper is organized as follows.
Section~\ref{masIP} outlines the general methodology for using exact rational
integer programming together with input/output verification for machine-assisted
theorem proving.
Section~\ref{IP} describes the IP formulations that we use to model Chvátal's
conjecture and presents valid inequalities for the underlying polytopes and
``cuts'' from the literature that reduce the number of integral solutions to the
IP formulations.
Section~\ref{experimentalResults}
contains a detailed description of our experimental results and
Section~\ref{conclusion} concludes with an outlook on future work.
The implementation and results are freely available to the
public~\cite{EiflerGit}.

\section{Verifiable Proofs for Integer Programming Results}\label{masIP}

In the following, we outline our computational methodology used to solve the
integer programs presented in Section~\ref{IP} such that the results can be
trusted and both input and output can be verified independently of the IP solver
used.
Figure \ref{fig:scheme} illustrates the four components.

\usetikzlibrary{shapes.geometric,backgrounds,calc,arrows,positioning-plus,node-families}
\tikzset{
  basic box/.style = {
    minimum height = 3.5cm,
    minimum width = 3.5cm,
    shape = rectangle,
    align = center,
    draw  = #1,
    fill  = #1!25,
    rounded corners},
  header node/.style = {
    Minimum Width = header nodes,
    font          = \strut\large\ttfamily,
    text depth    = +0pt,
    fill          = white,
    draw},
  header/.style = {%
    inner ysep = +1.5em,
    append after command = {
      \pgfextra{\let\TikZlastnode\tikzlastnode}
      node [header node] (header-\TikZlastnode) at (\TikZlastnode.north) {#1}
      node [span = (\TikZlastnode)(header-\TikZlastnode)]
        at (fit bounding box) (h-\TikZlastnode) {}
    }
  },
  hv/.style = {to path = {-|(\tikztotarget)\tikztonodes}},
  vh/.style = {to path = {|-(\tikztotarget)\tikztonodes}},
  fat blue line/.style = {ultra thick, blue}
}

\begin{figure}[ht]
\centering
  \scalebox{.8}{
    \begin{tikzpicture}[node distance = 0.5cm, thick, nodes = {align = center},
      >=latex]

    \node[Minimum Width = 4cm, fill = white, rounded corners] (vip)
    {VIPR to verify \\ branch-and-bound};

    \node[Minimum Width = 4cm, fill = white, below = of vip, rounded corners] (coq)
    {Data checker to\\ verify input};

    \begin{scope}[on background layer]
      \node[fit = (coq)(vip), basic box = blue,
      header = Verification] (verify) {};
    \end{scope}

    \node[Minimum Width = 4cm, fill = white, rounded corners, , left =  1.5cm  of verify] (sol)
    {Solve IP with exact \\ SCIP};
    \begin{scope}[on background layer]
      \node[fit = (sol), basic box = blue,
      header = Solving] (solb) {};
    \end{scope}

    \node[Minimum Width = 4cm, fill = white, rounded corners, , left =  1.5cm  of solb] (mod)
    {Model problem \\ as integer program};
    \begin{scope}[on background layer]
      \node[fit = (mod), basic box = blue,
      header = Modeling] (modb) {};
    \end{scope}

  \path[fat blue line, ->](modb.east) edge (solb.west);

  \path[fat blue line, ->, transform canvas={yshift=2.1em}](solb.east) edge (verify.west);
  \path[fat blue line, ->, transform canvas={yshift=-2.1em}](solb.east) edge (verify.west);

\end{tikzpicture}
}
\caption{General framework for modeling, solving, and verifying the results of integer programs.}
  \label{fig:scheme}
\end{figure}
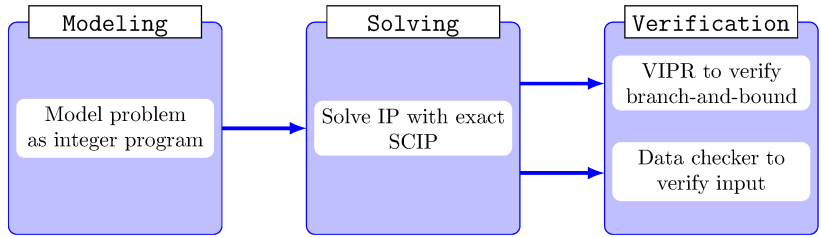

\paragraph{Modeling.}
As the first step, we use the modeling language ZIMPL~\cite{Koch2004} to
formulate the integer program.
ZIMPL employs exact rational arithmetic when instantiating the model in order to
ensure that no roundoff errors are introduced before passing the model to a
solver.

\paragraph{Solving.}
Next, we solve the IP using the exact rational variant of the MIP solver
SCIP~\cite{CookKochSteffyetal.2013}.
Exact SCIP implements a hybrid branch-and-bound algorithm that combines
floating-point and exact rational arithmetic in a safe manner.
Several methods are used in order to obtain safe dual bounds by correcting
relaxation solutions from fast floating-point linear programming (LP) solvers.
An exact rational LP solver, QSopt\_ex~\cite{APPLEGATE2007}, is used as
sparingly as possible.
Although exact SCIP still lacks many more sophisticated techniques implemented
in state-of-the-art floating-point solvers such as presolving reductions,
cutting planes, or symmetry handling, its design helps to yield superior
performance compared to a na\"ive branch-and-bound method solely relying on
rational LP solves.

\paragraph{Output Verification.}
Although exact SCIP is designed to provide safe results, the correctness of the
algorithm and implementation cannot easily be verified externally.
To address this issue, we use VIPR~\cite{VIPR}, a recently developed certificate
format that consists of the problem definition followed by an encoding of the
\bandb proof as a list of valid inequalities.
It rests on three simple inference steps that allow for elementary, stepwise
verification: aggregation of inequalities, rounding of right-hand sides, and
resolution of a binary disjunction.
In this sense, a VIPR certificate can, in theory, be checked by hand, although in
practice this may be prohibitive for larger certificates.
Hence, the VIPR project comes with an automatic, standalone checker, but the
simplicity of the format allows for the implementation of alternative checkers.

Exact SCIP can be configured to generate VIPR certificates during the solving
process such that its result must not be trusted blindly.
Its correctness can be verified completely independently of the solving process.


\paragraph{Input Verification.}

VIPR verification only ensures the correctness of the \bandb certificate with
respect to the integer program encoded in the problem section of the certificate
file.
However, due to implementation errors, the problem section of the certificate
file may actually not match the integer program of interest.
Therefore we implemented a safe input-checker that internally creates
its own representation of the constraint matrix for Problem $\redprob{n}$. It then
reads the problem section of the certificate file and checks if the
two constraint matrices coincide.
This input checker is written using the Coq proof assistant
\cite{coq2018}, a mathematical proof management system.
The matrix creation in this input checker is problem-specific, nevertheless it
can easily be adapted to formulations for similar problems.



All in all, we are confident that this framework ensures a high level of trust
in the computational proof of Theorem~\ref{7result}.
All tools are made publicly available for review~\cite{EiflerGit}, including the certificate
files for the computational results presented in
Section~\ref{experimentalResults}.

\section{A Polyhedral Approach to Chv\'{a}tal's Conjecture}\label{IP}


Chvátal's conjecture is a well-known open problem in extremal set
theory from 1974, later earning a spot among Erd{\H{o}}s' favorite
combinatorial problems~\cite{erdHos1981combinatorial}. Despite its
popularity, research efforts have yielded limited progress, mostly
restricted to special cases and related variants of the original
conjecture. Before continuing in more detail we need the following
definitions.

Let $[n]:=\left\{1,2, \ldots,n\right\}$. A \emph{family} $\mathcal{F}$
is a set of subsets of $[n]$. Let $U(\mathcal{F})$ denote the union of
all sets in $\mathcal{F}$. A family $\mathcal{F}$ is a \emph{downset}
if and only if $A \in \mathcal{F}$ and $B \subseteq A$ implies
$B \in \mathcal{F} $. If $\mathcal{F}$ is a downset then
$F \in \mathcal{F}$ is a \emph{base} if and only if no strict
supersets of $F$ are contained, \ie $F\subseteq D$ and
$D \in \mathcal{F}$ implies $F=D$.  A family $\mathcal{F}$ is called
\emph{intersecting} if and only if the intersection of any pairwise
sets in $\mathcal{F}$ is nonempty.  A family $\mathcal{F}$ is a
\emph{star} if and only there exists an element in $U(\mathcal{F})$
contained in all sets of $\mathcal{F}$.  A family $\mathcal{F}$ has
the \emph{star property} if and only if some maximum-sized
intersecting family in $\mathcal{F}$ is a star.  We are ready to state
Chvátal's conjecture as follows:

\begin{conjecture}[Chvátal~\cite{chvatal1974intersecting}]\label{chvatalsconj}
  Every downset has the star property.
\end{conjecture}

Sch\"onheim~\cite{schonheim1975hereditary} showed that Chvátal's
conjecture holds for all downsets $\mathcal{D}$ whose bases have a
nonempty intersection. Stein~\cite{stein1983chvatal} proved the
conjecture holds for all downsets in which all but one of the bases is
a \emph{simple} star, \ie a star in which the intersection of all of
its sets is equal to the intersection of any of its two
sets. Mikl{\'o}s~\cite{miklos1984great} showed that Chvátal's
conjecture holds for any downset $\mathcal{D}$ that contains an
intersection family of size
$\floor{\mathcal{D}/2}$. Sterboul~\cite{sterboul1974conjecture} proved
that any downsets whose sets have three or less elements always
satisfy Conjecture~\ref{chvatalsconj}. The last result was recently
proven again in different ways by Czabarka, Hurlbert,
Kamat~\cite{2017arXiv170300494C} and Olarte, Santos,
Spreer~\cite{OLARTE20192192}. Furthermore, Chvátal maintains a
website
dedicated to the conjecture with a substantial list of publications on
the topic~\cite{chvatalweb}.

Our main result on Chvátal's conjecture is the following:
\begin{theorem}\label{7result}
  Conjecture~\ref{chvatalsconj} holds for all downsets $\mathcal{D}$
  such that $|U(\mathcal{D})| \leq 7$.
\end{theorem}

To the best of our knowledge there
is no known computational methodology in the literature that
investigates Conjecture~\ref{chvatalsconj} even for small ground sets.
The previously known best bound on the cardinality of the ground set was $|U(\mathcal{D})| = 5$,
which follows directly from~\cite{sterboul1974conjecture} as we show in
Proposition~\ref{3sets}.

Already Fishburn~\cite{doi:10.1137/1030133} highlighted the connection
between combinatorial optimization and Chvátal's conjecture and
investigated related problems. Thus modeling the conjecture as an IP
and using solvers that safeguard against numerical issues is a natural
step to further investigate Conjecture~\ref{chvatalsconj}.
%
%
Furthermore, relying on an IP framework to investigate Chvátal's
conjecture for small ground sets has other advantages. First, the rich
and well-developed theory of polyhedral combinatorics, that is
inherent in an IP approach, may lead to new insights on
Conjecture~\ref{chvatalsconj}. Second, as we see in
Section~\ref{validIneq}, known
partial results on Chvátal's conjecture can be encoded as ``cuts'' in
our framework.  This improves the performance of the exact rational
solver and may allow strengthening of Theorem~\ref{7result} in the
future.

In the following sections we develop integer programming formulations for
Chvátal's conjecture over fixed-size ground sets.
The formulations are based on decision variables that index members of the power
set. Due to the exponential nature of power sets, the size of the formulations
is bound to grow quickly with the size of the ground set.
However, even for small ground sets little is known and the results in
Section~\ref{experimentalResults} improve upon what is known today.

\subsection{An Infeasibility-Based Formulation}
\label{infeasibilityproblem}

The first IP model is formulated such that Chvátal's conjecture holds for the considered ground set if and only if the IP has no solution.
In other words, a feasible solution to the IP formulation for any fixed $n$
would yield a counterexample to Chvátal's conjecture.
Let $2^{[n]}$ denote the power set of $[n]$, then we consider the integer
program \infprob{n},
\begin{subequations}
  \begin{align}
    \max  \sum_{S \in 2^{[n]} }x_{S}&   \label{eq:objective}\\
    x_{T} &\leq x_{S}
          && {\forall T \in 2^{[n]}, \forall S \in 2^{[n]}: S \subset T}, \label{eq:ideal}\\
    y_{T} + y_{S} &\leq 1
          && {\forall T \in 2^{[n]} \setminus \{\emptyset\}, \forall S \in 2^{[n]} \setminus \{\emptyset\}: T \cap S = \emptyset}, \label{eq:inters}\\
    y_S &\leq x_{S} &&\forall S \in  2^{[n]}, \label{eq:contain}\\
    \sum_{S \in 2^{[n]}: i \in S }x_{S} + 1 &\leq \sum_{S \in 2^{[n]}
                                              \setminus \{\emptyset\}
                                              }y_{S}
          && {\forall i \in [n]}, \label{eq:star}\\
    x_S, y_S &\in \{0,1\} && \forall S\in 2^{[n]}. \notag
  \end{align}
\end{subequations}

Here, $x$ encodes the set family $\incset(x) := \{ S \subseteq [n] : x_S = 1 \}$
and $y$ encodes the sub family $\incset(y) := \{ S \subseteq [n] : y_S = 1 \}$.
The first class of \emph{downset} inequalities \eqref{eq:ideal} ensures that
$\incset(x)$ is a downset.
The second class of \emph{intersecting} inequalities
\eqref{eq:inters} ensures that $\incset(y) \setminus \{\emptyset\}$ is an
intersecting family.
The third class of \emph{containment} inequalities \eqref{eq:contain} ensures
that the intersecting family is contained in the chosen downset, $\incset(y)
\subseteq \incset(x)$.
Finally, the fourth class of \emph{star} inequalities \eqref{eq:star} requires
that the intersecting family has greater cardinality than any star in the
downset.

\begin{theorem}\label{infeastheorem}
  Let $n$ be a positive integer. All downsets $\mathcal{F}$ such that
  $|U(\mathcal{F})|\leq n$ satisfy Chvátal's conjecture if and only if
  \infprob{n} is infeasible.
\end{theorem}
\begin{proof}
  Fix $n$.
  Suppose that Chvátal's conjecture does not hold, \ie there exists a downset
  $\mathcal D$ and an intersecting family $\mathcal Y \subseteq \mathcal{D}$ such that
  $\Abs{\mathcal Y}$ is larger than the size of every star in $\mathcal D$.
  \wlogU assume $\mathcal D \subseteq 2^{[n]}$.
  Let $x$ and $y$ be their incidence vectors, \ie $\mathcal{D} = \incset(x)$
  and $\mathcal{Y} = \incset(y)$.
  By construction, $x$ and $y$ satisfy
  constraints~(\ref{eq:ideal}--\ref{eq:contain}).
  Furthermore, for each element~$i\in[n]$, $\Abs{\mathcal Y}$ is larger than the
  size of all stars that have~$i$ as common element, hence \eqref{eq:star} is
  equally satisfied.
  In total, $x$ and $y$ constitute a feasible solution to \infprob{n}.

  Conversely, suppose all downsets $\mathcal{D}$ such that $|U(\mathcal{D})|\leq
  n$ satisfy Chvátal's conjecture.  Suppose $x$ and $y$ are feasible solutions
  to \infprob{n}.
  By \eqref{eq:ideal}, $\incset(x)$ forms a downset and
  $|U(\incset(x))|\leq n$.
  By \eqref{eq:inters} and \eqref{eq:contain}, $\mathcal{Y} := \incset(y) \setminus \{\emptyset\}$ forms an
  intersecting family contained in $\incset(x)$.
  Hence, $|\mathcal{Y}|$ can be at most the size of the largest star contained in
  $\incset(x)$,
  \begin{equation}
    |\mathcal{Y}| = \sum_{S \in 2^{[n]} \setminus \{\emptyset\} }y_{S} \leq
    \max_{i\in[n]} \sum_{S \in 2^{[n]}: i \in S }x_{S}.
  \end{equation}
  But then constraint \eqref{eq:star} is violated for $i_0 \in
  \arg\max_{i\in[n]} \sum_{S \in 2^{[n]}: i \in S }x_{S}$.
\end{proof}

Note that the objective function \eqref{eq:objective} of \infprob{n} is, in some sense, arbitrary
since we only need to decide whether the integer program has a feasible solution
or not.
The objective function encodes the cardinality of~$\incset(x)$, hence solving
\infprob{n} amounts to searching for a largest counterexample to
Conjecture~\ref{chvatalsconj}.
In the following we present a more advanced formulation that uses the optimal
value as an essential component.

\subsection{An Optimality-Based Formulation}
\label{feasibilityprob}

As already noted in~\cite{OLARTE20192192} it is sufficient to only consider
downsets generated by an intersecting family. We use this insight in the
following, advanced formulation \optprob{n},
\begin{subequations}
  \begin{align}
    \max  \sum_{S \in 2^{[n]}\setminus \{\emptyset\} }& y_{S} - z  \\
    y_{T} + y_{S} &\leq 1
                  && {\forall T \in 2^{[n]} \setminus \{\emptyset\}, \forall S \in 2^{[n]} \setminus \{\emptyset\}: T \cap S = \emptyset},  \label{eq:inters2}\\
    \sum_{S \in 2^{[n]}: i \in S }x_{S} &\leq z
                  && {\forall i \in [n]},  \label{eq:star2} \\
    y_T &\leq x_{S} &&  {\forall T \in 2^{[n]}, \forall S \in 2^{[n]}: S \subseteq T}, \label{eq:gen}\\
    x_S, y_S &\in \{0,1\}
                  && \forall S\in 2^{[n]}, \notag\\
    z &\in \mathbb{Z}_{\geq 0}. \notag
  \end{align}
\end{subequations}

The first class of \emph{intersecting} inequalities \eqref{eq:inters2}
is the same as \eqref{eq:inters}, whereas the second class of
\emph{star} inequalities \eqref{eq:star2} differs from
\eqref{eq:star}. It ensures that the largest star is bounded above by
the positive integer variable $z$. Finally, the third class of
\emph{generation} inequalities \eqref{eq:gen} ensures, as will be made
clearer in the proof of Theorem~\ref{feasthrm}, that an optimal
solution of \optprob{n} considers only downsets generated by the
intersecting family. We note that the generation inequalities
\eqref{eq:gen} can also be included in \infprob{n} instead of
\eqref{eq:ideal} and \eqref{eq:contain} by the same argument.  Before
we formally state and prove the correctness of \optprob{n} with
regards to Conjecture~\ref{chvatalsconj}, we need the following
observation.
\begin{observation}\label{tightz}
  Let $n$ be a positive integer. An optimal solution of \optprob{n}
  satisfies at least one star inequality~\eqref{eq:star2} with equality.
\end{observation}
Observation~\ref{tightz} follows from the objective function of
\optprob{n}.
Variable~$z$ is restricted only from below by its lower bound zero and the
left-hand sides of constraints~\eqref{eq:star2}.
Since \optprob{n} is a maximization problem and the objective coefficient of~$z$
is negative, in an optimal solution the variable $z$ will be as small as
possible.
This implies that at least one star inequality~\eqref{eq:star2} is tight.


\begin{theorem}\label{feasthrm}
  Let $n$ be a positive integer. Downsets $\mathcal{D}$ such that
  $|U(\mathcal{D})|\leq n$ satisfy Chvátal's conjecture if and only if
  the objective function value of an optimal solution of \optprob{n}
  is zero.
\end{theorem}
\begin{proof}
  Fix $n \in \N$.
  First note that \optprob{n} is feasible since any star is also an intersecting
  family.  Thus for any downset $\mathcal{D} \subseteq 2^{[n]}$, choosing $x$ as
  the indicator vector of $\mathcal{D}$ and setting the $y$-variables such that
  they represent a maximum-cardinality star in $\mathcal{D}$ yields a feasible
  solution.  Choosing $z$ to be the maximum star cardinality, \ie the smallest
  value such that constraints~\eqref{eq:star2} are satisfied, also proves a
  lower bound of zero on the objective value.

  Now suppose all downsets $\mathcal{D}$ such that $|U(\mathcal{D})|\leq n$
  satisfy Chvátal's conjecture and let $x,y,z$ be an optimal solution for
  \optprob{n}.
  Then it suffices to show that $\sum_{S \in 2^{[n]}\setminus \{\emptyset\} }
  y_{S} \leq z$.
  Constraints~\eqref{eq:inters2} ensure that $\mathcal Y := \incset(y) \setminus
  \{\emptyset\}$ forms an intersecting family.
  Furthermore, the $x$-variables do not appear in the objective function and are
  bounded below only by constraints~\eqref{eq:gen}.
  Hence, \wlogL we may assume that $x_S = \max_{T \supseteq S} y_T$.
  Then, by constraints~\eqref{eq:gen}, $\mathcal D := \incset(x)$ is a downset
  and $\mathcal Y \subseteq \mathcal D$.
  Assuming Chvátal's conjecture ensures that $|\mathcal Y| = \sum_{S \in
    2^{[n]}\setminus \{\emptyset\} } y_{S}$ is at most the size of the largest
  star in~$\mathcal D$, which by constraints~\eqref{eq:star2} is less than or
  equal to $z$.

  Conversely, suppose there exists a counterexample to Chvátal's conjecture,
  \ie a downset $\mathcal{D}$, $|U(\mathcal{D})| \leq n$, and an intersecting
  family $\mathcal{Y} \subseteq \mathcal{D}$ such that $|\mathcal Y|$ is larger
  than the size of any star in $\mathcal{D}$.
  \wlogU assume $\mathcal{D} \subseteq 2^{[n]}$ and let $x$ and $y$ be the
  incidence vectors of $\mathcal{D}$ and $\mathcal{Y}$, respectively, \ie
  $\mathcal{D} = \incset(x)$ and $\mathcal{Y} = \incset(y)$.
  Let $z$ be the size of the largest star in $\mathcal{D}$, \ie $z = \max
  \sum_{S \in 2^{[n]}: i \in S }x_{S}$.
  Then by construction, $x,y,z$ is a feasible solution for \optprob{n}.
  Because we consider a counterexample, the objective function value is
  at least one.
\end{proof}

\subsection{Valid Inequalities and Model Reductions}
\label{validIneq}

As mentioned in Section~\ref{intro}, one of the advantages of an IP
approach is that \infprob{n} and \optprob{n} can be studied in greater
depth through polyhedral combinatorial techniques. Furthermore known
results from the literature can be expressed as valid inequalities and problem reductions for
\infprob{n} and \optprob{n}, in the sense that
Theorems~\ref{infeastheorem} and~\ref{feasthrm} still hold with these
additional constraints, and the number of feasible solutions is less
than or equal to the number of current solutions.
This is demonstrated in the following section and may help to increase the size
of $n$ for which the models can be solved.

First, consider the intersecting inequalities of form~\eqref{eq:inters}
and~\eqref{eq:inters2}.
When $T = [n] \setminus S$, these can be interpreted as a special case of
the following partition inequalities.

\begin{proposition}
  Let $n$ be a positive integer and let $\mathcal{P}$ be a partition
  of $[n]$. Then the inequality
  \begin{equation}
    \sum_{S \in \mathcal{P} }y_{S} \leq 1
  \end{equation}
  is a valid for \infprob{n} and \optprob{n}.
\end{proposition}
\begin{proof}
  Suppose $\sum_{S \in \mathcal{P} }y_{S} \geq 2$ for an integer feasible
  solution $y$, then there exist $S,T \in \mathcal P$ with $y_{S} = y_{T} = 1$,
  violating~\eqref{eq:inters} and~\eqref{eq:inters2}.
\end{proof}

%
As a consequence, intersection inequalities that do
not cover the whole ground set can be strengthened.

\begin{proposition}
  Let $n$ be a positive integer and
  $S,T \in 2^{[n]}\setminus \{\emptyset\}$ such that
  $T \cap S = \emptyset$. Suppose $S \cup T \neq [n]$. Then the
  intersecting inequality
  \begin{equation}\label{eq:weakinters}
    y_{T}+y_{S} \leq 1
  \end{equation}
  is dominated by a partition inequality.
\end{proposition}
\begin{proof}
  $S,T$ can be completed to a partition by their complement $[n] \setminus (S
  \cup T)$. Inequality~\eqref{eq:weakinters} is trivially dominated by the
  corresponding partition inequality $y_{T}+y_{S}+y_{[n] \setminus (S \cup T)}
  \leq 1$.
\end{proof}

The large number of partitions prohibits the static addition of these
inequalities to the formulation.
However, modern IP solvers automatically extract the conflicting $y$-assignments
from constraints~\eqref{eq:inters} and~\eqref{eq:inters2} and add partition
inequalities dynamically.
Next, consider the following central result.

\begin{theorem}[Berge~\cite{berge1975theorem}]~\label{bergthrm} If
  $\mathcal{D}$ is a downset then $\mathcal{D}$ is a disjoint union of
  pairs of disjoint sets, together with $\emptyset$ if $|\mathcal{D}|$
  is odd.
\end{theorem}
This yields the following result, that can easily be expressed as a valid inequality for
\infprob{n} and \optprob{n}, as in Corollary~\ref{bergecut}.
\begin{corollary}[Anderson~\cite{anderson2002combinatorics} p.105]\label{bergecor}
  Let $\mathcal{D}$ be a downset and $\mathcal{Y}$ an intersecting
  family such that $\mathcal{Y} \subseteq \mathcal{D}$. Then
  $2|\mathcal{Y}| \leq |\mathcal{D}|$.
\end{corollary}
\begin{corollary}\label{bergecut}
  Let $n$ be any positive integer. Suppose the following inequality
  \begin{equation}
    \sum_{S \in 2^{[n]}\setminus \{\emptyset\}}2y_{S} \leq \sum_{S \in 2^{[n]}}x_{S}
  \end{equation}
  is added to \infprob{n} and \optprob{n}. Then Theorems~\ref{infeastheorem}
  and~\ref{feasthrm} hold for the modified formulations of \infprob{n} and
  \optprob{n}, respectively.
\end{corollary}
The following result is used in~\cite{OLARTE20192192} to give a
simple proof that Chvátal's conjecture holds for all downsets whose
sets have three elements or less.
\begin{theorem}[Kleitman, Magnanti~\cite{kleitman1972number}]\label{fixonetwo}
  Any intersecting family that is contained in the union of two stars
  generates a downset that satisfies Conjecture~\ref{chvatalsconj}.
\end{theorem}
As a consequence, $y$-variables for sets with one or two elements can be fixed
to zero in \infprob{n} and \optprob{n}.
\begin{corollary}\label{fixy12}
  Let $n$ be any positive integer. For \infprob{n} and \optprob{n} fix
  $y_S = 0$ for all $S \in 2^{[n]}$ such that $1 \leq |S| \leq
  2$. Then Theorems~\ref{infeastheorem} and~\ref{feasthrm} hold for
  \infprob{n} and \optprob{n}, respectively, with the given fixings.
\end{corollary}
\begin{proof}
  Suppose $y$ stems from a solution of \infprob{n} or \optprob{n}, then it
  encodes an intersecting family $\mathcal Y := \incset(y) \setminus
  \{\emptyset\}$.
  If $\{i\} \in \mathcal Y$, then $\mathcal Y$ is a star centered around $i$.
  If $\{i,j\} \in \mathcal Y$, then $\mathcal Y$ is the union of two stars
  centered around $i$ and $j$, respectively. By Theorem~\ref{fixonetwo}, these
  do not amount to counterexamples to Chvátal's conjecture and can safely be
  excluded from the formulations without changing the feasibility status of
  \infprob{n} and the optimal objective value of \optprob{n}, respectively.
\end{proof}

Theorem \ref{fixonetwo} can also be exploited for a short proof, that the conjecture holds 
for all ground sets of size less or equal than $5$.

\begin{proposition}\label{3sets}
  Conjecture~\ref{chvatalsconj} holds for all downsets $\mathcal{D}$
  such that $|U(\mathcal{D})| \leq 5  $.
\end{proposition}
\begin{proof}
  Consider an intersecting family $\mathcal{Y}$ such that $|U(\mathcal{Y})| \leq 5$,
  \wlogL $\mathcal{Y} \subseteq 2^{[5]}$.
  All $10$ sets of size $3$ have to be part of the intersecting family.
  If this is not the case, then $\mathcal Y$ is contained in the union of the
  stars of the remaining two elements and the conjecture holds by Theorem \ref{fixonetwo}.
  The maximal star, containing only elements of size $1,2,3$ has size~$11$. 

  Let us now consider sets of size $4$. There are $k$ sets of size $4$ in $\mathcal Y$ with 
  $0 \le k \le \binom{5}{4}$. For any star, there exists at most one set of size $4$ 
  that does not contain the common element. 
  Therefore the size of the largest intersecting family is at most $10+k$, whereas the size
  of the largest star is $11+k-1 = 10+k$.
  
  Since the whole ground set can not be part of the intersecting family, this concludes the proof.
\end{proof}

Variable fixings are certainly the most effective improvements to the problem
formulation, since they directly reduce the problem size as opposed to general
valid inequalities that increase the number of constraints.
If we know that Conjecture~\ref{chvatalsconj} holds for all downsets
$\mathcal{D}$ such that $|U(\mathcal{D})| \leq n$ for some fixed $n$, then we
can use a simple variable fixing scheme for the case when $|U(\mathcal{D})| = n
+1$, as follows.

\begin{proposition}\label{fixx1}
  Let $n$ be a fixed positive integer. Suppose \infprob{n} is infeasible
  and the objective function value of an optimal solution of
  \optprob{n} is zero for all positive integers $n_0 < n$. Fix
  $x_{S} =1$ for all $S\in 2^{[n]}$ such that $|S| = 1$. Then
  Theorems~\ref{infeastheorem} and~\ref{feasthrm} hold for \infprob{n} and
  \optprob{n}, respectively, with the given fixings.
\end{proposition}
\begin{proof}
  Consider the $x$-vector from a solution of \infprob{n} or \optprob{n} and
  suppose that $x_{\{i\}} = 0$ for some element $i \in [n]$.
  As in the proofs of Theorems~\ref{infeastheorem} and~\ref{feasthrm}, we may
  assume that $x$ encodes a downset $\mathcal D := \incset(x)$ and $\mathcal Y
  := \incset(y) \subseteq \mathcal D$ is an intersecting family.
  By the downset property, $x_S = 0$ for all $S \ni i$.
  But then $|U(\mathcal D)| < n$ and by assumption the solution is not a
  counterexample to Conjecture~\ref{chvatalsconj}.
  Hence, we may fix $x_{S} = 1$ for all $S\in 2^{[n]}$ such that $|S| = 1$.
\end{proof}

We can apply this proposition incrementally by starting from the case $n=6$, with all 
$x$-variables for sets that contain exactly one element fixed to $1$. After solving,
we increase $n$ by one and repeat the procedure.

Finally, we
discuss how to exploit the fundamental result of~\cite{sterboul1974conjecture} as
an additional fixing scheme.

\begin{proposition}\label{fixx2}
  Let $n \geq 6$ be a fixed integer. In \infprob{n} and \optprob{n} fix $x_{S} =
  1$ for all $S\in 2^{[4]}$.  Then Theorems~\ref{infeastheorem}
  and~\ref{feasthrm} hold, respectively, with the given fixings.
\end{proposition}
\begin{proof}
  According to \cite{sterboul1974conjecture}, counterexamples to Chvátal's
  conjecture must feature a downset that contains at least one set of size four.
  By permuting the elements in $[n]$ suitably, we can always ensure that this
  set is \{1,2,3,4\}.
\end{proof}

To summarize, we arrive at the following improved formulation $\redprob{n}$,
\begin{subequations}
  \begin{align}
    \max  \sum_{S \in 2^{[n]}\setminus \{\emptyset\} }& y_{S} - z  \\
    y_{T} + y_{S} &\leq 1
                  && {\forall T \in 2^{[n]} \setminus \{\emptyset\}, \forall S \in 2^{[n]} \setminus \{\emptyset\}: T \cap S = \emptyset},  \label{eq:inters3}\\
    \sum_{S \in 2^{[n]}: i \in S }x_{S} &\leq z
                  && {\forall i \in [n]},  \label{eq:star3} \\
    y_T &\leq x_{S} &&  {\forall T \in 2^{[n]}, \forall S \in 2^{[n]}: S \subseteq T}, \label{eq:gen3}\\
    \sum_{S \in 2^{[n]}\setminus \{\emptyset\}}2y_{S} &\leq \sum_{S \in 2^{[n]}}x_{S}, \label{eq:bergecut}\\
    y_S &= 0       && \forall S \in 2^{[n]}: 1 \leq |S| \leq 2, \label{eq:fixy12}\\
    x_S &= 1       && \forall S \in 2^{[n]}: |S| = 1, \label{eq:fixx1}\\
    x_S &= 1       && \forall S\subseteq [4], \label{eq:fixx2}\\
    x_S, y_S &\in \{0,1\}
                  && \forall S\in 2^{[n]}, \notag\\
    z &\in \mathbb{Z}_{\geq 0}. \notag
  \end{align}
\end{subequations}

This formulation serves as the basis for our proof of Theorem~\ref{7result} that
uses the following equivalence incrementally for $n = 6$ and $n = 7$.
We summarize our reductions in the following theorem.

\begin{theorem}\label{redthrm}
  Let $n$ be a positive integer and suppose Chvátal's conjecture holds for all
  downsets $\mathcal{D}$ such that $|U(\mathcal{D})|\leq n-1$.
  Then all downsets $\mathcal{D}$ such that $|U(\mathcal{D})|\leq n$ satisfy
  Chvátal's conjecture if and only if the objective function value of an optimal
  solution of $\redprob{n}$ is zero.
\end{theorem}
\begin{proof}
  As follows from Corollary~\ref{bergecut}, constraint~\eqref{eq:bergecut} is a
  valid inequality for \optprob{n}.
  Corollary~\ref{fixy12} shows that
  constraints~\eqref{eq:fixy12} do not exclude any counterexamples that may have
  objective function value greater than zero.
  According to Proposition~\ref{fixx1}, the same holds for
  constraints~\eqref{eq:fixx1} under the assumption that Chvátal's conjecture is
  correct for smaller ground sets.
  According to Proposition~\ref{fixx2}, constraints~\eqref{eq:fixx2} may exclude
  counterexamples, but only as long as at least one symmetric counterexample
  remains feasible.
\end{proof}

\section{Computational Results}
\label{experimentalResults}

Using the safe computational framework outlined in Section~\ref{masIP}, we could
solve \optprob{n} and $\redprob{n}$ for $n=5,6$, and $7$, producing a
machine-assisted proof of Theorem~\ref{7result}.
Furthermore, several floating-point MIP solvers could solve $\redprob{8}$ to
optimality.
Although this does not constitute a safe proof, it makes it highly likely that
Chvátal's conjecture holds for $n = 8$ and a search for counterexamples should
focus on larger ground sets.

Beyond the plain question of solvability, in this section we provide details
regarding the following questions: \emph{What are the times spent for solving
  the integer programs and how are they affected by the improvements in the
  formulations?  How large are the resulting certificates and how expensive is
  their verification?  How does the performance of the exact framework compare
  to the performance of standard floating-point MIP solvers?}

The results for the optimality-based formulations are provided in
Table~\ref{res-table}.
All tests were run on a cluster of computing nodes with Intel Xeon Gold 5122
CPUs with 3.6\,GHz and 96\,GB of main memory.
The exact version of SCIP was built with CPLEX 12.6.3 as floating-point LP
solver and QSopt\_ex 2.5.10~\cite{APPLEGATE2007} for exact rational LP solves.
As floating-point MIP solver, we used
SCIP 6.0.0~\cite{GleixnerBastubbeEifleretal.2018}, built with CPLEX 12.8.0 as
the underlying LP solver.
The time limit was set to 12 hours for all runs and on each computing node only
one job was executed at a time.

\begin{table}[ht]
  \caption{Computational details for solving the Chvátal IPs for the two
    formulations \optprob{n} and $\redprob{n}$.  The sizes reported for the VIPR
    certificates are for uncompressed text files.  The running time for input verification is negligible and always below 5~seconds.}
    \label{res-table}
    \begin{tabular*}{0.9\textwidth}{l@{\quad\extracolsep{\fill}}rrrrrrr}
    \toprule
       &     &       &       & SCIP 6.0.0 & SCIP exact & \multicolumn{2}{c}{VIPR} \\
    \cmidrule(l){5-5}\cmidrule(l){6-6}\cmidrule(l){7-8}
    IP & $n$ & \#vars & \#ineqs & time [s]   & time [s] & size [MB] & time [s] \\
    \midrule
    \midrule
    \optprob{n}
    & 5       & 63    & 427   & 0.2        & 0.5     & 0.5      & 0.07       \\
    & 6       & 127   & 1336  & 2.3        & 22.6    & 73       & 4.6        \\
    & 7       & 255   & 4125  & 91.0       & 4024.2  & 21000    & 1258.3     \\
    & 8       & 511   & 12618 & --         & --      & --       & --         \\
    \midrule
    $\redprob{n}$
    & 5       & 31    & 433   & 0.1        & 0.1     & 0.018    & 0.005      \\
    & 6       & 88    & 1317  & 0.1        & 0.2     & 0.25     & 0.2       \\
    & 7       & 208   & 4050  & 13.5       & 124.5   & 163      & 28.9       \\
    & 8       & 455   & 12424 & 7278.9     & --      & --       & --         \\
    \bottomrule
  \end{tabular*}
\end{table}

First, we observe the effectiveness of the additional inequalities and fixings
applied in $\redprob{n}$.
The running times of exact SCIP are significantly reduced, as are the sizes and
verification times for the VIPR certificates.
Furthermore, only $\redprob{8}$ can be solved by floating-point SCIP, while it
times out for \optprob{8}.
Second, note that the size of the IPs is not large compared to what MIP solvers
today can often handle easily in many industrial applications.
This underlines the difficulty of the underlying combinatorial question.
Third, the sizes and running times of checking the VIPR certificate are
significant but do not constitute a bottleneck for the current framework.
Note that the times for input verification in Coq are not reported, because they
are negligible and always below 5~seconds.


\section{Conclusion}\label{conclusion}

The goal of this paper was to bridge a gap that currently divides theory and
practice of integer programming.
In theory, the integer programming paradigm holds the promise of globally
optimal, proven results.
In practice, however, there is no established computational framework that
rigorously safeguards results that are obtained with integer programming
software against numerical and programming errors.

Bridging this gap necessarily involved theoretical and practical aspects.
On the one hand, we developed non-trivial integer programming formulations for
an exemplary application: Chvátal's conjecture, a long-standing open question in
extremal combinatorics, where even low-dimensional cases are unanswered.
One advantage of this approach was the flexibility of the IP formulations to
include partial results from the literature by valid inequalities and variable
fixings.
Solving these formulations, we could show that counterexamples to the conjecture
do not seem to exist for the cases $|U(\mathcal{D})| = 6$, $7$, and $8$, which
were previously open.

On the other hand, we had to develop and combine mathematical software in order
to equip these results with the level of numerical rigor and independent
verifiability that is required by computational proofs of mathematical
theorems.
Such verifiable computer proofs could be produced for the cases
$|U(\mathcal{D})| = 6$ and ~$7$.
A proof for the case $|U(\mathcal{D})| = 8$ is tangible if the performance of
the exact rational solver is enhanced in the future.
Hence, the results also motivate sustained work on closing the current
performance gap between exact and inexact MIP solvers.

We hope that the generality of the computational framework
presented makes it useful for the investigation of other open questions in
extremal combinatorics and beyond.
Most directly, our IP models could be appropriately modified to investigate
variations on Chvátal's conjecture such as proposed by Snevily~\cite{SnevilyWeb}
or a generalization of Chvátal's conjecture proposed by
Borg~\cite{borg2011chvatal}.

\bibliography{sample}

\end{document}